\documentclass[11pt,leqno]{amsart}
\usepackage{amsmath,amssymb,latexsym}
\usepackage{amscd}
\usepackage[OT4]{fontenc}
\usepackage{lmodern}
\usepackage{mathrsfs}
\usepackage{amsthm}
\usepackage{amsfonts}
\usepackage{url}
\usepackage{enumerate}
\usepackage{graphicx}
\usepackage{hyperref}
\usepackage[all]{xy}
\usepackage{color}
\usepackage[normalem]{ulem}
\usepackage{tikz}
\usepackage{enumerate}

 \usepackage[
    backend=biber,
    style=numeric,
  ]{biblatex}

\AtEveryBibitem{\clearfield{issn}}
\AtEveryBibitem{\clearfield{doi}}
\AtEveryBibitem{\clearfield{url}}
\AtEveryBibitem{\clearfield{isbn}}
\usetikzlibrary{matrix}
\usetikzlibrary{arrows}

\hypersetup{pdfstartview={XYZ null null 1.25}}

\newcommand{\C}{\mathbb C}
\newcommand{\R}{\mathbb R}
\newcommand{\N}{\mathbb N}

\newcommand{\K}{\mathbb K}

\newcommand{\cA}{\mathcal A}
\newcommand{\cB}{\mathcal B}
\newcommand{\cC}{\mathcal C}
\newcommand{\cD}{\mathcal D}

\newcommand{\cK}{\mathcal K}
\newcommand{\cL}{\mathcal L}

\newcommand{\cQ}{\mathcal Q}
\newcommand{\cM}{\mathcal M}
\newcommand{\cN}{\mathcal N}
\newcommand{\cS}{\mathcal S}
\newcommand{\cV}{\mathcal V}

\newcommand{\hk}{\hookrightarrow}

\newcommand{\HS}{\mathcal{HS}(\ell_2)}
\newcommand{\ldss}{\mathcal{L}(s',s)}
\newcommand{\ls}{\mathcal{L}(s)}
\newcommand{\lds}{\mathcal{L}(s')}
\newcommand{\lstars}{\mathcal{L}^*(s)}
\newcommand{\lslds}{\mathcal{L}(s)\cap\mathcal{L}(s')}
\renewcommand{\c}{{\colon}}
\newcommand{\op}{{\operatorname{ind}}}
\newcommand{\opp}{{\operatorname{proj}}}

\newcommand{\bproof}{{\raggedright\textbf{Proof.}} \ }
\newcommand{\bqed}{\hspace*{\fill} $\Box $\medskip}

\let\epsilon\varepsilon
\let\phi\varphi

\theoremstyle{plain}

\newtheorem{Prop}{Proposition}[section]
\newtheorem{Cor}[Prop]{Corollary}
\newtheorem{Lem}[Prop]{Lemma}
\newtheorem{Th}[Prop]{Theorem}

\theoremstyle{definition}
\newtheorem{Def}[Prop]{Definition}

\newtheoremstyle{mytheoremstyle} 
    {\topsep}                    
    {\topsep}                    
    {}                   
    {}                           
    {\bfseries}                   
    {.}                          
    {.5em}                       
    {}  

\theoremstyle{mytheoremstyle}
\newtheorem{Rem}[Prop]{Remark}

\title{{\sc The multiplier algebra of the noncommutative Schwartz space}}
\author{{\sc Tomasz Cia{\'s} and Krzysztof Piszczek}}
\date{}

\begin{document}

\begin{abstract}
We describe the multiplier algebra of the noncommutative Schwartz space. 
This multiplier algebra can be seen as the largest ${}^*$-algebra of unbounded operators on a separable Hilbert space with the classical Schwartz space of rapidly decreasing functions as the domain.
We show in particular that it is neither a $\cQ$-algebra nor $m$-convex. On the other hand, we prove that classical tools of functional analysis, for example, the closed graph theorem, the open mapping theorem or the uniform boundedness principle, are still available.
\end{abstract}

\footnotetext[3]{{\em
Acknowledgement.} The research of both authors has been supported by the National Center of Science, Poland,
grant no. DEC-2013/10/A/ST1/00091.}

\maketitle

\section{Introduction}

The aim of this paper is to study algebraic and topological properties of some specific topological algebra with involution, called the \emph{multiplier algebra of the noncommutative Schwartz space} and denoted by $\cM\cS$. In particular we will show that, in spite of the fact that $\cM\cS$ is neither Banach nor metrizable, the Closed Graph Theorem, Open Mapping Theorem and Uniform Boundedness Principle work on this space -- see Th. \ref{classical-fa}. Since it is a locally convex space the Hahn-Banach Theorem obviously holds.

The algebra $\cM\cS$ can be described in many ways.
For instance, it can be seen as the
``intersection'' $\cL(\cS(\R))\cap\cL(\cS'(\R))$, where $\cL(\cS(\R))$ (resp., $\cL(\cS'(\R))$) is the algebra of continuous linear operators on the well-known Schwartz space
$\cS(\R)$ of smooth rapidly decreasing functions (resp., on the space $\cS'(\R)$ of tempered distributions).
It appears that we can replace the above function and distribution spaces with the corresponding sequence spaces.
This means that $\cM\cS$ is isomorphic (as a topological ${}^*$-algebra) to the 
``intersection'' $\lslds$, where $s$ is the space of rapidly decreasing sequences, $s'$ is the space of slowly increasing sequences
(for definitions see next section) and $\ls$, $\lds$ are the corresponding spaces of continuous linear operators.
The 
``intersection'' -- whatever it is -- allows us to introduce a natural involution and this is why we consider $\lslds$ instead of, for example, $\ls$.
The algebra $\cM\cS$ also turns out to be the maximal $\mathcal{O}^*$-algebra with domain $s$, and thus it can be viewed as the largest algebra
of unbounded operators on $\ell_2$ with domain $s$ -- see the book of K. Schm\"udgen \cite{Sch} which deals with $\mathcal{O}^*$-algebras, and especially \cite[Part I.2]{Sch}.
Moreover, by Proposition \ref{prop_lslds=LambdaA}, our algebra is isomorphic (as a topological ${}^*$-algebra) to the matrix algebra
\[\Lambda(\cA):=\bigg\{x=(x_{i,j})_{i,j\in\N}\in\C^{\N^2}:\,\,\forall\,N\in\N_0\,\,\exists\,n\in\N_0\c\,\,\sup_{i,j\in\N}|x_{i,j}|\max\bigg\{\frac{i^N}{j^n},\frac{j^N}{i^n}\bigg\}<\infty\bigg\}.\]

As the name suggests, $\cM\cS$ can be also viewed as the algebra of 
``multipliers'' for the so-called \emph{noncommutative Schwartz space} $\cS$, also known as the \emph{algebra of smooth operators}.
The noncommutative Schwartz space is a specific $m$-convex Fr\'echet $^*$-algebra isomorphic to several, operator algebras naturally appearing in analysis,
for example, to the algebra $\cL(\cS'(\R),\cS(\R))$ of continuous linear operators from $\cS'(\R)$ to $\cS(\R)$, the algebra $\ldss$ or to the algebra
\[\cK_\infty:=\{(x_{i,j})_{i,j\in\N}\in\C^{\N^2}:\,\,\forall\,N\in\N_0\,\,\sup_{i,j\in\N}|x_{i,j}|i^Nj^N<\infty\}\]
of \emph{rapidly decreasing matrices} (see \cite[Th. 1.1]{PD} for more representations).
It is also isomorphic -- topologically but not as an algebra -- to the Fr\'echet space $\cS(\R)$ (which explains its name).

The noncommutative Schwartz space and the space $\cS(\R)$ itself play a role in a number of fields, for example: structure theory of Fr\'echet spaces and splitting of short exact sequences (see \cite[Part IV]{MV});
K-theory (see \cite{JC}); $C^*$-dynamical systems (see \cite{ENN}); cyclic cohomology for crossed products (see \cite{LS});  operator analogues for locally convex spaces (see \cite{EW,EW1}) and
quantum mechanics, where it is called the \textit{space of physical states} and its dual is the so-called \textit{space of observables} (see \cite{DH}). Recently, some progress in the investigation
of the noncommutative Schwartz space has been made. This contains: functional calculus (see \cite{TC}); description of closed, commutative, $^*$-subalgebras (see \cite{TC1});
automatic continuity (see \cite{KP}); amenability properties (see \cite{AP,KP}); Grothendieck inequality (see \cite{KP4}).

However, the significance of the algbra $\cM\cS$ lies not only in the fact that it is a muliplier algebra of some well-known algebra of operators but also in its resemblance to the $C^*$-algebra
$\cB(\ell_2)$ of bounded operators on $\ell_2$ in the context of some class of Fr\'echet and topological ${}^*$-algebras. This can be already seen in the very definition of the maximal $\mathcal{O}^*$-algebra
on the domain $s$ (which is, recall, isomorphic to $\cM\cS$). It is also worth pointing out that $\cS$ is considered as a Fr\'echet analogon of the algebra $\cK(\ell_2)$ of compact operators on $\ell_2$
and $\cB(\ell_2)$ is the multiplier
algebra of $\cK(\ell_2)$ whereas, as we have already noted, $\cM\cS$ is the multiplier algebra of $\cS$ -- in fact, in order to prove this, we apply methods used in the case of $\cK(\ell_2)$ and $\cB(\ell_2)$
(compare Theorem \ref{th_Ls=DC} with \cite[Prop. 2.5]{Busby} and \cite[Example 3.1.2]{Mur}). Moreover, it seems that $\cM\cS$ contains as closed ${}^*$-subalgebras many important topological ${}^*$-algebras,
e.g. Fr\'echet algebras $C^{\infty}(M)$ of smooth functions on each compact smooth manifold $M$ (the proof of this fact will be presented in a forthcoming paper).

Finally, we find quite interesting to compare some commutative sequence ${}^*$-algebras with the corresponding noncommutative operator ${}^*$-algebras. This is done by the
following diagram with the horizontal continuous embeddings of algebras:
\begin{displaymath}
\xymatrix{
s \ar @{^{(}->}[r] \ar @{<~>}[d] & \ell_1 \ar@{^{(}->}[r]\ar @{<~>}[d] & \ell_2 \ar@{^{(}->}[r]\ar@{<~>}[d] & c_0 \ar@{^{(}->}[r]\ar @{<~>}[d] & \ell_\infty \ar@{^{(}->}[r]\ar @{<~>}[d] & s' \ar @{<~>}[d] \\
\cS \ar @{^{(}->}[r]              & \cN(\ell_2) \ar@{^{(}->}[r]            & \HS    \ar@{^{(}->}[r]       & \cK(\ell_2) \ar@{^{(}->}[r] & \cB(\ell_2) & \cM\cS,
}
\end{displaymath}
where $\cN(\ell_2)$ and $\HS$ is the algebra of nuclear and Hilbert-Schmidt operators, respectively.
The ``vertical correspondences'' from $s$ up to $c_0$ mean, for example, that every monotonical sequence of nonnegative numbers belonging to a commutative algebra from the first row is a sequence
of singular numbers of some operator of its noncommutative analogues.
Moreover, algebras from the first row are embedded in a canonical way, as the algebras of diagonal operators, into the corresponding algebras from the second row. It will be made clear below that $\cB(\ell_2)$ is not embedded in $\cM\cS$ and vice versa.

%

This article is divided into four parts. Section 2 recalls basic notation and properties of the objects involved. In Section 3, we describe the multiplier algebra of the noncommutative Schwartz space.
Section 4 deals with its topological properties and, finally, in Section 5 we consider properties of $\cM\cS$ as a topological algebra.
(For more information on functional analysis, see \cite{J,MV}, for more on Banach algebra theory, see \cite{D}; and for non-Banach operator algebras, see \cite{Sch}.)

\section{Notation and Preliminaries}

In what follows, we set

\begin{align}
\N:=\{1,2,3,\ldots\}, \nonumber \\
\N_0:=\{0,1,2,\ldots\}. \nonumber
\end{align}
For locally convex spaces $E$ and $F$, we denote by $\cL(E,F)$ the space of all continuous linear operators from $E$ to $F$, and we set $\cL(E):=\cL(E,E)$. These spaces will be considered with the topology $\tau_{\cL(E,F)}$ of uniform convergence on bounded sets.

By $s$ we denote the \emph{space of rapidly decreasing sequences}\index{space!of rapidly decreasing sequences}, that is, the Fr\'echet space
\[s:=\bigg\{\xi=(\xi_j)_{j\in\N}\in\C^\N\c|\xi|_n:=\bigg(\sum_{j=1}^\infty|\xi_j|^2j^{2n}\bigg)^{1/2}<\infty\text{ for all }n\in\N_0\bigg\}\]
with the topology given by the system $(|\cdot|_n)_{n\in\N_0}$ of norms. We will denote by $s_n$ the Hilbert space corresponding to the norm $|\cdot|_n$.

By \cite[Prop. 27.13]{MV} we may identify the strong dual of $s$, that is, the space of all continuous linear functionals on $s$ with the topology of uniform convergence on bounded
subsets of $s$, with the \emph{space of slowly increasing sequences}
\[s':=\bigg\{\xi=(\xi_j)_{j\in\N}\in\C^\N\c|\xi|_{-n}:=\bigg(\sum_{j=1}^\infty|\xi_j|^2j^{-2n}\bigg)^{1/2}<\infty\text{ for some }n\in\N_0\bigg\}\]
equipped with the inductive limit topology for the sequence $(s_{-n})_{n\in\N_0}$, where $s_{-n}$ is the Hilbert space corresponding to the norm $|\cdot|_{-n}$.
In other words, the locally convex topology on $s'$ is given by the family $\{|\cdot|'_{B}\}_{B\in\cB}$ of norms , $|\xi|'_{B}:=\sup_{\eta\in B}|\langle \eta,\xi\rangle|$,
where $\cB$ denotes the class of all bounded subsets of $s$.

The space $\cL(s',s)$ is a Fr\'echet space, whose topology is described by the sequence $(\|\cdot\|_n)_{n\in\N_0}$ of norms,
\[\|x\|_n:=\sup\{|x\xi|_n\c\,\,\xi\in U_n^{\circ}\}\label{norms},\]
where $U_n^{\circ}=\{\xi\in s'\c\,\,|\xi|_{-n}\leqslant1\}$ are polars of the zeroneighbourhood basis $(U_n)_{n\in\N_0}$ in $s$. One can show that $\cS$ is in fact isomorphic (as a Fr\'echet space) to the space $s$ (see \cite[\S41, 7.(5)]{K2} and \cite[Lemma 31.1]{MV}). Since $\cL(s',s)\hk\cB(\ell_2)$ (continuous embedding given by $\ell_2\hk s'\overset{x}{\to}s\hk\ell_2$), we can endow our space with multiplication, involution and order structure of the $C^*$-algebra $\cB(\ell_2)$. With these operations $\cL(s',s)$ becomes an $m$-convex Fr\'echet $^*$-algebra and is called the \textit{noncommutative Schwartz space} or the \textit{algebra of smooth operators}.

We devote a large part of the present article to considering the spaces $\ls$, $\lds$ and their ``intersection"
\[\lslds:=\{x\in\ls:x=\widetilde{x}\mid_s\textrm{ for some (and hence unique) }\widetilde{x}\in\lds\}\]
equipped with the topology $\tau_{\lslds}:=\tau_{\ls}\cap\tau_{\lds}$; $\tau_{\lslds}$ is therefore determined by the family $\{q_{n,B}\}_{n\in\N_0,B\in\cB}$ of seminorms, where

\begin{equation}
q_{n,B}(x):=\max\bigg\{\sup_{\xi\in B}|x\xi|_n,\sup_{\xi\in U_n^\circ}|x\xi|'_{B}\bigg\}
\label{qnB}
\end{equation}
and $\cB$ is the class of all bounded subsets of $s$. It is easy to show that $\ls$ and $\lds$ are topological algebras and, as we will show in Proposition \ref{prop_lstars=lslds_as_top_star_alg}, the same is true for $\lslds$.

The spaces $\ls$ and $\lds$ can be seen as the completed tensor products $s'\widehat{\otimes}s$ and $s\widehat{\otimes}s'$, respectively -- see \cite[\S43, 3.(7)]{K2}.
Since $s$ and $s'$ are nuclear (see \cite[Prop. 28.16]{MV}), the injective and the projective tensor product topologies coincide (see \cite[21.2, Th. 1]{J}).
Consequently (see \cite[15.4, Th. 2 \& 15.5, Cor. 4]{J}), $\ls$ and $\lds$ admit natural $\mathrm{PLS}$-topologies (see below) coinciding with $\tau_{\cL(s)}$ and $\tau_{\cL(s')}$, that is,

\begin{equation}
\begin{split}
\ls=\operatorname{proj}_{N\in\N_0}\operatorname{ind}_{n\in\N_0}\cL(s_n,s_N), \\
\lds=\operatorname{proj}_{N\in\N_0}\operatorname{ind}_{n\in\N_0}\cL(s'_N,s'_n).
\label{operator-representation}
\end{split}
\end{equation}

%

\begin{Rem}
To be precise, the above two representations of $\cL(s)$ and $\cL(s')$ are not of $\mathrm{PLS}$-type. But this can be easily overcome following the first part of Remark \ref{rem_PLS}.
\end{Rem}

As for the class of $\mathrm{PLS}$-spaces (the best reference for which is \cite{PD-PLS} and the references therein) we will need more information.
Recall that by a \textit{$\mathrm{PLB}$-space} we mean a locally convex space $X$ whose topology is given by
\[X=\opp_{N\in\N_0}\op_{n\in\N_0}X_{N,n},\]
where all the $X_{N,n}$'s are Banach spaces and all the linking maps $\iota_{N,n}^{N,n+1}\c X_{N,n}\hk X_{N,n+1}$ are linear and continuous inclusions.
If all the $X_{N,n}$'s are Hilbert spaces then we call $X$ a \textit{$\mathrm{PLH}$-space}. If the linking maps $(\iota_{N,n}^{N,n+1})_{N,n\in\N_0}$ are compact (nuclear)
then $X$ is called a \textit{$\mathrm{PLS}$-space} (\textit{$\mathrm{PLN}$-space}).

Of particular importance for us will be the so-called \textit{K\"{o}the-type $\mathrm{PLB}$-spaces}. Recall that a \textit{K\"{o}the $\mathrm{PLB}$-matrix}
is a matrix $\cC:=(c_{j,N,n})_{j\in\N,N,n\in\N_0}$ of nonnegative scalars satisfying:
\begin{enumerate}
\item[(i)]
$\forall\,j\in\N\,\,\exists\,N\in\N_0\,\,\forall\,n\in\N_0\c\,\,\,\,c_{j,N,n}>0$,
\item[(ii)]
$\forall\,j\in\N,N,n\in\N_0\c\,\,\,\,c_{j,N,n+1}\leqslant c_{j,N,n}\leqslant c_{j,N+1,n}$.
\end{enumerate}
We define
\[\Lambda^p(\cC):=\{x=(x_j)_j|\,\,\forall\,N\in\N_0\,\,\exists\,n\in\N_0\c\,\,||x||_{N,n,p}<+\infty\}\,\,\,\,\,\,\,\,(1\leqslant p<\infty),\]
where
\[||x||_{N,n,p}:=\Bigl(\sum_{j=1}^{\infty}(|x_j|c_{j,N,n})^p\Bigr)^{1/p}.\]
For $p=\infty$ we use the respective `sup' norms. Then $\Lambda^p(\cC)$ can be identified with the space $\opp_{N\in\N_0}\op_{n\in\N_0}\ell^p(c_{j,N,n})$ and it is called
a K\"{o}the-type $\mathrm{PLB}$-space. If for all $N,n\in\N_0$ we have $\lim\limits_j\frac{c_{j,N,n+1}}{c_{j,N,n}}=0$ (compact linking maps) then $\Lambda^p(\cC)$ is a $\mathrm{PLS}$-space and
if for all $N,n\in\N_0$ we have $\sum_j\frac{c_{j,N,n+1}}{c_{j,N,n}}<\infty$ (nuclear linking maps) then $\Lambda^p(\cC)$ is a $\mathrm{PLN}$-space.
If $\Lambda^p(\cC)=\Lambda^q(\cC)$, as sets, for all $1\leqslant p,q\leqslant\infty$ then we simply write $\Lambda(\cC)$.

For later use, we distinguish specific K\"{o}the $\mathrm{PLB}$-matrices on $\N\times\N$, defined as follows:

\begin{equation}
\begin{split}
\cB=(b_{ij;N,n})_{i,j\in\N,N,n\in\N_0},\,\,\,\,\,\,\,\,b_{ij;N,n}:=\frac{i^N}{j^n}, \\
\cB'=(b'_{ij;N,n})_{i,j\in\N,N,n\in\N_0},\,\,\,\,\,\,\,\,b'_{ij;N,n}:=\frac{j^N}{i^n}
\label{plb-matrices}
\end{split}
\end{equation}

and

\begin{equation}
\cA=(a_{ij;N,n})_{i,j\in\N,N,n\in\N_0},\,\,\,\,\,\,\,\,a_{ij;N,n}:=\max\{b_{ij;N,n},b'_{ij;N,n}\}=\max\bigg\{\frac{i^N}{j^n},\frac{j^N}{i^n}\bigg\}.
\label{plb-matrix}
\end{equation}

\begin{Prop}
\label{lp-equal}
Let $\cM$ be any of the K\"{o}the-type $\mathrm{PLB}$-matrices given by \eqref{plb-matrices} or \eqref{plb-matrix}. Then $\Lambda^p(\cM)=\Lambda^q(\cM)$ topologically, for all $1\leqslant p,q\leqslant\infty$.
\end{Prop}

\bproof
Since $m_{ij;N+2,n}=(ij)^2m_{ij;N,n+2}$ for all indices $i,j\in\N,N,n\in\N_0$, we have the inequalities
\[\|\cdot\|_{N,n,\infty}\leqslant\|\cdot\|_{N,n,1}\,\,\,\,\text{and}\,\,\,\,\|\cdot\|_{N,n+2,1}\leqslant C\|\cdot\|_{N+2,n,\infty},\]
which hold for all $N,n\in\N_0$ with $C=\sum_{i,j=1}^{\infty}(ij)^{-2}=\frac{\pi^4}{36}$.
\bqed

Using the notion of K\"{o}the-type $\mathrm{PLB}$-spaces we get from \eqref{operator-representation} and Proposition \ref{lp-equal} topological isomorphisms:

\begin{equation}
\begin{split}
\cL(s)\cong\Lambda(\cB)\hspace{20pt}\text{and}\hspace{20pt}\cL(s')\cong\Lambda(\cB').
\label{matrix-representation}
\end{split}
\end{equation}
In both cases the isomorphism is given by $x\mapsto(\langle xe_j,e_i\rangle)_{i,j\in\N}$.

%
%
%
%
%

\section{Representations of the multiplier algebra}

In this section we want to describe the so-called multiplier algebra of $\cS$ which is, in some sense, the largest algebra of operators acting on $\cS$.
The algebra $\lslds$ seems to be a good candidate, because if $x\in\cS$ and $y\in\lslds$, then clearly $xy,yx\in\cS$. Now, using heuristic arguments, we may show that the algebra $\lslds$ is optimal. Assume that $y\in\cL(E,F)$ for some locally convex spaces $E,F$. If $xy\in\cS$ for every $x\in\cS$ then, in particular,
$(\langle\cdot,\xi\rangle\xi)y\in\cS$ for all $\xi\in s$, and therefore
$\langle y(\eta),\xi\rangle$ has to be well-defined for every $\xi\in s$ and $\eta\in s'$, which shows that $y\colon s'\to s'$.
Similarly, we show that if $yx\in\cS$ for every $x\in\cS$ then $y\colon s\to s$. Hence, $y\in\lslds$.

Another, more abstract, approach to multipliers goes through the so-called double centralizers
(see Definition \ref{def_double_centralizer}) and is due to Johnson \cite{Json}.
Theorem \ref{th_Ls=DC} below shows that this approach leads again to $\lslds$.
In Corollary \ref{cor_DC=lstars=lslds=LambdaA}, we show that the multiplier algebra of $\cS$ has other representations -- also important in further investigation.

The theory of double centralizers of $C^*$-algebras was developed by Busby (see \cite{Busby} and also \cite[pp. 38--39, 81--83]{Mur}).
Our exposition for the noncommutative Schwartz space will follow that of $C^*$-algebras.

\begin{Def}\label{def_double_centralizer}
Let $A$ be a ${}^*$-algebra over $\C$. A pair $(L,R)$ of maps from $A$ to $A$ (neither linearity nor continuity is required) such
that $xL(y)=R(x)y$ for $x,y\in A$ is called a \emph{double centralizer}\index{double centralizer} on $A$. We denote the set of all double centralizers on $A$ by $\mathcal{DC}(A)$.
Moreover, for a map $T\colon A\to A$, we define $T^*\colon A\to A$ by $T^*(x):=(T(x^*))^*$.

Now, let $(L_1,R_1),(L_2,R_2)\in\mathcal{DC}(A)$, $\lambda\in\C$. We define:
\begin{enumerate}[\upshape(i)]
 \item $(L_1,R_1)+(L_2,R_2):=(L_1+L_2,R_1+R_2)$;
 \item $\lambda(L_1,R_1):=(\lambda L_1,\lambda R_1)$;
 \item $(L_1,R_1)\cdot(L_2,R_2):=(L_1L_2,R_2R_1)$;
 \item $(L_1,R_1)^*:=(R_1^*,L_1^*)$.
\end{enumerate}
\end{Def}

A straightforward computation shows that $\mathcal{DC}(A)$ with the operations defined above is a ${}^*$-algebra.
The elements of $A$ correspond to the elements of $\mathcal{DC}(A)$ via the map,
called the \emph{double representation} of $A$ (see \cite[p. 301]{Json}),

\begin{equation}
\varrho\colon A\to\mathcal{DC}(A),\quad\varrho(x):=(L_x,R_x),
\label{double-reprpesentation}
\end{equation}
where $L_x(y):=xy$ and $R_x(y):=yx$ are the left and right multiplication maps, respectively. One can easily show that $\varrho$ is a homomorphism
of ${}^*$-algebras.

\begin{Def}\label{def_faithful}
Let $A$ be an algebra over $\C$ and let $I$ be an ideal in $A$.
\begin{enumerate}
 \item[\upshape(i)] We say that $A$ is \emph{faithful} if for every $x\in A$ we have: $xA=\{0\}$ implies $x=0$ and $Ax=\{0\}$ implies $x=0$.
 \item[\upshape(ii)] An ideal $I$ is called \emph{essential} in $A$ if for every $x\in A$ the following implications hold: if $xI=\{0\}$ then $x=0$ and if $Ix=\{0\}$ then $x=0$.
\end{enumerate}
\end{Def}

It is well-known that every $C^*$-algebra is an essential ideal in its multiplier algebra \cite[p. 82]{Mur}), so it is faithful (see also \cite[Cor. 2.4]{Busby}). For example, the $C^*$-algebra of
compact operators on $\ell_2$ is an essential ideal in $\cL(\ell_2)$. The analogue of this result holds for $\cS$ and $\lslds$.

\begin{Prop}\label{prop_ldss_faithful}
An algebra $\cS$ is an essential ideal in $\lslds$. In particular, $\cS$ is faithful.
\end{Prop}
\bproof Clearly, $\cS$ is an ideal in $\lslds$. Assume that $xz=0$ for all $z\in \cS$. Then, in particular, for $\xi\in s$ and $z:=\langle\cdot,\xi\rangle\xi$, we get
$\langle\cdot,\xi\rangle x(\xi)=0$. Thus $x(\xi)=0$ for all $\xi\in s$, i.e $x=0$. This gives $x\cS=\{0\}\Rightarrow x=0$. Applying now the involution we get the implication
$\cS x=\{0\}\Rightarrow x=0$.
\bqed

The following result can be deduced from \cite[Th. 7 \& Th. 14]{Json}. For the convenience of the reader we present a more direct proof.
We follow the proof of \cite[Prop. 2.5]{Busby} (the case of $C^*$-algebras).

\begin{Prop}\label{prop_(L,R)_L(s',s)}
Let $A$ be a faithful $m$-convex Fr\'echet algebra and let $(L,R)\in\mathcal{DC}(A)$. Then
\begin{enumerate}[\upshape(i)]
 \item $L$ and  $R$ are linear continuous maps on $A$;
 \item $L(xy)=L(x)y$ for every $x,y\in A$;
 \item $R(xy)=xR(y)$ for every $x,y\in A$.
\end{enumerate}
\end{Prop}
\bproof
(i) Let $x,y,z\in A$, $\alpha,\beta\in\C$. Then
\[zL(\alpha x+\beta y)=R(z)(\alpha x+\beta y)=\alpha R(z)x+\beta R(z)y=z(\alpha L(x)+\beta L(y)),\]
and since $A$ is faithful, $L(\alpha x+\beta y)=\alpha L(x)+\beta L(y)$.

Now, let $(x_j)_{j\in\N}\subset A$ and assume that $x_j\to0$ and $L(x_j)\to y$ (convergence in the topology of $A$).
Let $(||\cdot||_q)_{q\in\N_0}$ be a fundamental system of submultiplicative seminorms on $A$. Then
\begin{align*}
||zy||_q&\leqslant||zy-zL(x_j)||_q+||zL(x_j)||_q=||z(y-L(x_j))||_q+||R(z)x_j||_q\\
&\leqslant||z||_q\cdot||y-L(x_j)||_q+||R(z)||_q\cdot||x_j||_q\to0,
\end{align*}
as $j\to\infty$, so $||zy||_q=0$ for every $q\in\N_0$, and therefore $zy=0$. Hence, by the assumption on $A$, $y=0$. Now, by the Closed Graph Theorem for
Fr\'echet spaces (see e.g. \cite[Th. 24.31]{MV}), $L$ is continuous. Analogous arguments work for the map $R$.

(ii) Let $x,y,z\in A$. Then
\[zL(xy)=R(z)xy=(R(z)x)y=(zL(x))y=z(L(x)y),\]
and therefore, $L(xy)=L(x)y$.

(iii) Analogously as in (ii).
\bqed

Now, we need to prove that elements of $\lslds$ can be seen as some unbounded operators on $\ell_2$, namely as the elements of the class
\[\cL^*(s):=\{x\colon s\to s: x\textrm{ is linear},s\subset\mathcal{D}(x^*)\textrm{ and }x^*(s)\subset s\}.\]
Here
\[\cD(x^*):=\{\eta\in\ell_2:\exists\zeta\in\ell_2\;\forall\xi\in s\quad\langle x\xi,\eta\rangle=\langle\xi,\zeta\rangle\}\]
and $x^*\eta:=\zeta$ for $\eta\in\cD(x^*)$ (one can show that such a vector $\zeta$ is unique). This $^*$-operation defines a natural involution on $\cL^*(s)$.
$\cL^*(s)$ is known as the maximal $\mathcal{O}^*$-algebra with domain $s$ and it can be viewed as the largest ${}^*$-algebra of unbounded operators
on $\ell_2$ with domain $s$ (see \cite[I.2.1]{Sch} for details). Let $\cB$ denote the set of all bounded subsets of $s$.
We endow $\lstars$ with the topology $\tau_{\lstars}$ given by the family $\{p_{n,B}\}_{n\in\N_0,B\in\cB}$ of seminorms,

\begin{equation}
p_{n,B}(x):=\max\{\sup_{\xi\in B}|x\xi|_n,\sup_{\xi\in B}|x^*\xi|_n\}.
\label{pnB}
\end{equation}
It is well-known that $\lstars$ is a complete locally convex space and a topological ${}^*$-algebra (see \cite[Prop. 3.3.15]{Sch} and Remark \ref{rem_lstars} below).

\begin{Rem}\label{rem_lstars}
The so-called graph topology on $s$ of the $\mathcal{O}^*$-algebra $\lstars$ is given by the system of seminorms $(||\cdot||_a)_{a\in\lstars}$, $||\xi||_a:=||a\xi||_{\ell_2}$ (\cite[Def. 2.1.1]{Sch}).
It is easy to see that the usual Fr\'echet space topology on $s$ is equal to the graph topology (consider the diagonal map $s\ni(\xi_j)_{j\in\N}\mapsto(j^n\xi_j)_{j\in\N}\in s$, $n\in\N_0$),
and therefore the topology $\tau_{\lstars}$ on $\lstars$ coincides with the topology $\tau^*$ (see \cite[pp. 81--82]{Sch}) defined by the seminorms $\{p^{a,B}\}_{a\in\lstars,B\in\cB}$,
\[p^{a,B}(x):=\max\Big\{\sup_{\xi\in B}||ax\xi||_{\ell_2},\,\sup_{\xi\in B}||ax^*\xi||_{\ell_2}\Big\}.\]
\end{Rem}


\begin{Lem}\label{lem_extension_of_op_lstars}
We have $\lstars\subset\ls$ and every operator $x\in\lstars$ can be extended to an operator $\widetilde{x}\in\lds$. Moreover,
\[\langle x^*\eta,\xi\rangle=\langle\eta,\widetilde{x}\xi\rangle\]
for all $\xi\in s'$ and $\eta\in s$.
\end{Lem}

\bproof
Take $x\in\cL^*(s)$. Let $(\xi_j)_{j\in\N}\subset s$ and assume that $\xi_j\to 0$ and $x\xi_j\to\eta$ as $j\to\infty$.
Then, for every $\zeta\in s$, we have
\[\langle x\xi_j,\zeta\rangle=\langle\xi_j,x^*\zeta\rangle\to 0\]
and, on the other hand,
\[\langle x\xi_j,\zeta\rangle\to\langle\eta,\zeta\rangle.\]
Hence $\langle\eta,\zeta\rangle=0$ for every $\zeta\in s$, and therefore $\eta=0$.
By the Closed Graph Theorem for Fr\'echet spaces, $x\colon s\to s$ is continuous, and thus $\lstars\subset\ls$.

Fix $x\in\lstars$, $\xi\in s'$, and define a linear
functional $\phi_\xi\colon s\to\C$, $\phi_\xi(\eta):=\langle x^*\eta,\xi\rangle$. From the continuity of $x^*\colon s\to s$, it follows that
for every $q\in\N_0$ there is $r\in\N_0$ and $C>0$ such that $|x^*\eta|_q\leqslant C|\eta|_r$ for all $\eta$ in $s$. Hence, with the same quantifiers,
we get
\begin{equation}\label{eq_cont_phi_xi(eta)}
|\phi_\xi(\eta)|=|\langle x^*\eta,\xi\rangle|\leqslant|x^*\eta|_q\cdot|\xi|_{-q}\leqslant C|\eta|_r\cdot|\xi|_{-q},
\end{equation}
and thus $\phi_\xi$ is continuous.
Consequently, for each $\xi\in s'$ we can find a unique $\zeta\in s'$ such that
\[\langle\eta,\zeta\rangle=\phi_\xi(\eta)=\langle x^*\eta,\xi\rangle\]
for all $\eta\in s$ and we may define $\widetilde{x}\colon s'\to s'$ by $\widetilde{x}\xi:=\zeta$.
Clearly, $\widetilde{x}$ is a linear extension of $x$, and moreover $\widetilde x$ is continuous. Indeed, by (\ref{eq_cont_phi_xi(eta)}),
for every $q\in\N_0$ there is $r\in\N_0$ and $C>0$ such that
\[|\widetilde x\xi|_{-r}=\sup_{|\eta|_r\leqslant1}|\langle\eta,\widetilde x\xi\rangle|=\sup_{|\eta|_r\leqslant1}|\langle x^*\eta,\xi\rangle|\leqslant C|\xi|_{-q}\]
for all $\xi\in s'$. This shows the continuity of $\widetilde x$, and the proof is complete.
\bqed

The following result follows also from \cite[Prop. 2.2]{Kur}.

\begin{Prop}\label{prop_lstars=lslds}
$\cL^*(s)=\ls\cap\lds$ as sets.
\end{Prop}

\bproof
The inclusion $\cL^*(s)\subset\ls\cap\lds$ follows directly from Lemma \ref{lem_extension_of_op_lstars}.

Let $x\in\lslds$. For each $\eta\in s$ we define a linear functional $\psi_{\eta}\colon s'\to \C$,
$\psi_{\eta}(\xi):=\langle\widetilde x\xi,\eta\rangle$, where $\widetilde x\colon s'\to s'$ is the continuous extension of $x$.
By the continuity of the operator $\widetilde{x}$ and the Grothendieck Factorization Theorem \cite[24.33]{MV}, it follows that for every $r\in\N_0$ there is $q\in\N_0$ and $C>0$ such that
$|\widetilde{x}\xi|_{-q}\leqslant C|\xi|_{-r}$ for $\xi\in s'$. Hence, for $\xi\in s'$, we have
\[|\psi_\eta(\xi)|=|\langle\widetilde x\xi,\eta\rangle|\leqslant|\widetilde x\xi|_{-q}\cdot|\eta|_q\leqslant C|\eta|_q\cdot|\xi|_{-r}.\]
This shows that $\psi_\eta$ is continuous, and therefore there exists $\zeta\in s$ such that $\psi_\eta(\cdot)=\langle\cdot,\zeta\rangle$.
Consequently, $\langle x\xi,\eta\rangle=\langle\xi,\zeta\rangle$
for $\xi\in s$, hence $s\subset\mathcal{D}(x^*)$ and $x^*(s)\subset s$, that is, $x\in\cL^*(s)$.
\bqed

If we endow $\lslds$ with the involution on $\lstars$, we obtain that:

\begin{Prop}\label{prop_lstars=lslds_as_top_star_alg}
$\cL^*(s)=\ls\cap\lds$ as locally convex spaces and ${}^*$-algebras. Consequently, $\lslds$ is a complete locally convex space and a topological ${}^*$-algebra.
\end{Prop}

\bproof
By definition and Proposition \ref{prop_lstars=lslds}, $\cL^*(s)=\ls\cap\lds$ as ${}^*$-algebras. Let us compare fundamental systems
$\{p_{n,B}\}_{n\in\N_0,B\in\cB}$ \eqref{pnB} and $\{q_{n,B}\}_{n\in\N_0,B\in\cB}$ \eqref{qnB} of seminorms on $\cL^*(s)$ and $\ls\cap\lds$, respectively.

Take $x\in\lstars$ with its unique extension $\widetilde{x}\in\lds$. Let $B$ be a bounded subset of $s$ and $n\in\N_0$. Then, by Lemma \ref{lem_extension_of_op_lstars},
\begin{equation}\label{eq_pnB=qnB}
\sup_{\eta\in B}|x^*\eta|_n=\sup_{\eta\in B}\sup_{\xi\in U^\circ_n}|\langle x^*\eta,\xi\rangle|=\sup_{\eta\in B}\sup_{\xi\in U^\circ_n}|\langle \eta,\widetilde{x}\xi\rangle|
=\sup_{\xi\in U_n^\circ}|\widetilde{x}\xi|'_B.
\end{equation}
This shows that $p_{n,B}(x)=q_{n,B}(x)$.

Since $\lstars$ is a complete locally convex space and a topological ${}^*$-algebra (see \cite[Prop. 3.3.15]{Sch} and Remark \ref{rem_lstars}), the result follows.
\bqed

By Propositions \ref{prop_ldss_faithful} and \ref{prop_(L,R)_L(s',s)}, $\mathcal{DC}(\cS)\subset\cL(\cS)\times\cL(\cS)$, and thus we may endow $\mathcal{DC}(\cS)$ with
the corresponding subspace topology, denoted by $\tau_{\mathcal{DC}(\cS)}$. Note that a typical continuous seminorm of an element $(L,R)$ of $\cL(\cS)\times\cL(\cS)$ is given by the formula
$\max\{\sup_{x\in M}||L(x)||_n,\sup_{x\in M}||R(x)||_n\}$, where $M$ is a bounded subset of $\cS$ and $n\in\N_0$.

For $u\in\lslds$ we define the left and right multiplication maps $L_u, R_u\colon \cS\to \cS$, $L_u(x):=ux$, $R_u(x):=x\widetilde{u}$, where $\widetilde{u}\colon s'\to s'$
is the extension of $u$ according to the definition of $\lslds$.

Theorem \ref{th_Ls=DC} below states that the double representation \eqref{double-reprpesentation} of $\cS$ can be extended to an isomorphism of topological ${}^*$-algebras
$\lslds$ and $\mathcal{DC}(\cS)$, and thus $\lslds$ can be seen as the multiplier algebra of $\cS$ (compare with \cite[Th. 3.1.8 \& Example 3.1.2]{Mur}).

\begin{Th}\label{th_Ls=DC}
The map $\widetilde{\varrho}\colon\lslds\to\mathcal{DC}(\cS)$, $u\mapsto (L_u,R_u)$ is an isomorphism of locally convex spaces and ${}^*$-algebras. Consequently, $\mathcal{DC}(\cS)$
is a complete locally convex space and topological ${}^*$-algebra.
\end{Th}
\bproof
Throughout the proof, for $\xi,\eta\in s$, $\xi\otimes\eta$ denotes the one-dimensional operator $\langle\cdot,\eta\rangle\xi\colon s'\to s$.

Clearly, for $u\in\lslds$, the left and right multiplication maps
$L_u,R_u\colon\cS\to\cS$ are well defined. Moreover, it is easy to see that $xL_u(y)=R_u(x)y$ for $x,y\in\cS$ and $u\in\lslds$.
Hence, $(L_u,R_u)\in\mathcal{DC}(\cS)$ for every $u\in\lslds$, that is, $\widetilde{\varrho}$ is well defined.

The proof of the fact that $\widetilde{\varrho}$ is a ${}^*$-algebra homomorphism is straightforward and $\widetilde{\varrho}$ is injective, because $\cS$ is an essential ideal in $\lslds$
(Proposition \ref{prop_ldss_faithful}). We will show that $\widetilde{\varrho}$ is surjective.

Let $(L,R)\in\mathcal{DC}(\cS)$ and fix $e\in s$ with $||e||_{\ell_2}=1$. We define a linear continuous map
(use Propositions \ref{prop_ldss_faithful} and \ref{prop_(L,R)_L(s',s)}) $u\colon s\to s$ by
\[u\xi:=L(\xi\otimes e)(e).\]
For $\xi,\eta\in s$ we have
\begin{align}\label{eq_uxi_eta}
\begin{split}
\langle u\xi,\eta\rangle&=\langle L(\xi\otimes e)(e),\eta\rangle=\langle L(\xi\otimes e)(e),(\eta\otimes e)(e)\rangle=\langle (e\otimes \eta)[L(\xi\otimes e)(e)],e\rangle\\
&=\langle [(e\otimes \eta)L(\xi\otimes e)](e),e\rangle=\langle [R(e\otimes \eta)(\xi\otimes e)](e),e\rangle=\langle R(e\otimes \eta)[(\xi\otimes e)(e)],e\rangle\\
&=\langle R(e\otimes \eta)(\xi),e\rangle=\langle \xi,(R(e\otimes \eta))^*(e)\rangle.
\end{split}
\end{align}
This means that $u^*\eta=(R(e\otimes \eta))^*(e)\in s$ for $\eta\in s$. Hence, $s\subset\mathcal{D}(u^*)$ and $u^*(s)\subset s$.
Consequently, $u\in\cL^*(s)$, and thus, by Proposition \ref{prop_lstars=lslds}, $u$ has the continuous extension $\widetilde u\colon s'\to s'$.

By Propositions \ref{prop_ldss_faithful} and \ref{prop_(L,R)_L(s',s)}, for $\zeta\in s$ we obtain
\begin{align*}
L_u(\xi\otimes\eta)(\zeta)&=(u\xi\otimes\eta)(\zeta)=[L(\xi\otimes e)(e)\otimes\eta](\zeta)=\langle\zeta,\eta\rangle L(\xi\otimes e)(e)\\
&=L(\xi\otimes e)(\langle\zeta,\eta\rangle e)=L(\xi\otimes e)[(e\otimes\eta)(\zeta)]=[L(\xi\otimes e)(e\otimes\eta)](\zeta)\\
&=L((\xi\otimes e)(e\otimes\eta))(\zeta)=L(\xi\otimes\eta)(\zeta),
\end{align*}
hence $L_u(\xi\otimes\eta)=L(\xi\otimes\eta)$. Since $\{\xi\otimes\eta:\xi,\eta\in s\}$ is linearly dense in $\cS$, it follows that $L_u=L$.
Likewise, (\ref{eq_uxi_eta}) implies that, for $\zeta\in s$,
\begin{align*}
R_u(\xi\otimes\eta)(\zeta)&=[(\xi\otimes\eta)\widetilde u](\zeta)=\langle u\zeta,\eta\rangle\xi=\langle R(e\otimes \eta)(\zeta),e\rangle\xi
=(\xi\otimes e)((R(e\otimes \eta)(\zeta))=\\
&=[(\xi\otimes e)R(e\otimes \eta)](\zeta)=R((\xi\otimes e)(e\otimes \eta))(\zeta)=R(\xi\otimes\eta)(\zeta),
\end{align*}
and therefore $R_u=R$. Hence $\widetilde{\varrho}(u)=(L_u,R_u)=(L,R)$, and thus $\widetilde{\varrho}$ is surjective.


Next, we shall prove that $\widetilde{\varrho}$ is continuous. Let $M$ be a bounded subset of $\cS$ and let $n\in\N_0$.
Since the involution on $\cS$ is continuous (see \cite[p. 148]{TC}), the set $M^*$ is bounded and there are $C>0$, $k\geqslant n$ such that $||y^*||_n\leqslant C||y||_k$ for all $y\in\cS$.
Define
\begin{align*}
B_1&:=\{x\xi\colon x\in M, \xi\in s', |\xi|_{-n}\leqslant1\},\\
B_2&:=\{x^*\xi\colon x\in M, \xi\in s', |\xi|_{-k}\leqslant1\}.
\end{align*}
Then, for all $m\geqslant k$, we have
\begin{align*}
\sup\{|\eta|_m\colon \eta\in B_1\}&\leqslant\sup_{x\in M}||x||_m<\infty,\\
\sup\{|\eta|_m\colon \eta\in B_2\}&\leqslant\sup_{x\in M^*}||x||_m<\infty,
\end{align*}
and therefore $B_1$ and $B_2$ are bounded subsets of $s$.
Now,
\[\sup_{x\in M}||L_u(x)||_n=\sup_{x\in M}||ux||_n=\sup_{x\in M}\sup_{|\xi|_{-n}\leqslant1}|u(x\xi)|_n=\sup_{\eta\in B_1}|u(\eta)|_n\leqslant q_{n,B_1}(u).\]
Moreover, by Lemma \ref{lem_extension_of_op_lstars}, $(x\widetilde{u})^*=u^*x^*$ for $x\in\cS$ and $u\in\lslds$, and thus
\begin{align*}
\sup_{x\in M}||R_u(x)||_n&=\sup_{x\in M}||x\widetilde{u}||_n\leqslant C\sup_{x\in M}||(x\widetilde{u})^*||_k=C\sup_{x\in M}||u^*x^*||_k=\sup_{x\in M}\sup_{|\xi|_{-k}\leqslant1}|u^*(x^*\xi)|_k\\
&=\sup_{\eta\in B_2}|u^*\eta|_k\leqslant p_{k,B_2}(u)=q_{k,B_2}(u),
\end{align*}
where the last identity follows from (\ref{eq_pnB=qnB}). Consequently,
\[\max\bigg\{\sup_{x\in M}||L_u(x)||_n,\sup_{x\in M}||R_u(x)||_n\bigg\}\leqslant\max\{q_{n,B_1}(u),q_{k,B_2}(u)\}\leqslant q_{k,B_1\cup B_2}(u),\]
and thus $\widetilde{\varrho}$ is continuous.

Finally, we show that the inverse of $\widetilde{\varrho}$ is continuous. Let us take a bounded subset $B$ of $s$ and $n\in\N_0$. Define
\[M:=\{x_\xi\colon\xi\in B\setminus\{0\}\}\cup\{\mathbf{0}\},\]
where $x_\xi:=||\xi||_{\ell_2}^{-1}\xi\otimes\xi$ and $\mathbf{0}$ is the zero operator in $\cS$. For all $m\in\N_0$, we have
\[\sup_{x\in M}||x||_m=\sup_{\xi\in B\setminus\{0\}}||\xi||_{\ell_2}^{-1}|\xi|_m^2\leqslant\sup_{\xi\in B\setminus\{0\}}|\xi|_{2m}<\infty,\]
where the first inequality follows from the Cauchy-Schwartz inequality and the last one is a consequence of the boundedness of the set $B$. Hence $M$ is a bounded subset of $\cS$.

Let
\[B':=\{x\xi\colon x\in M, \xi\in s',|\xi|_{-n}\leqslant1\}.\]
Clearly, $0\in B'$. If $\xi\in B\setminus\{0\}$, then $\xi=x_\xi(||\xi||_{\ell_2}^{-1}\xi)$ and $\big|\,||\xi||_{\ell_2}^{-1}\xi\big|_{-n}\leqslant1$, hence $\xi\in B'$. Consequently, $B\subset B'$.
Again by identity (\ref{eq_pnB=qnB}), we get, for all $u\in\lslds$,
\begin{align*}
q_{n,B}(u)&=p_{n,B}(u)=\max\bigg\{\sup_{\eta\in B}|u\eta|_n,\sup_{\eta\in B}|u^*\eta|_n\bigg\}\leqslant\max\bigg\{\sup_{\eta\in B'}|u\eta|_n,\sup_{\eta\in B'}|u^*\eta|_n\}\bigg\}\\
&=\max\bigg\{\sup_{x\in M}\sup_{|\xi|_{-n}\leqslant1}|u(x\xi)|_n,\sup_{x\in M}\sup_{|\xi|_{-n}\leqslant1}|u^*(x\xi)|_n\bigg\}=\max\bigg\{\sup_{x\in M}||ux||_n,\sup_{x\in M}||u^*x^*||_n\bigg\}\\
&\max\bigg\{\sup_{x\in M}||ux||_n,\sup_{x\in M}||x\widetilde{u}||_n\bigg\}=\max\bigg\{\sup_{x\in M}||L_u(x)||_n,\sup_{x\in M}||R_u(x)||_n\bigg\},
\end{align*}
and therefore ${\widetilde{\varrho}}^{-1}$ is continuous.
\bqed

\section{Topological properties of the multiplier algebra}

We start by showing how the multiplier algebra of $\cS$ can be realized as a matrix algebra. Before that we make rather easy but very efficient observation.

\begin{Prop}
\label{complemented}
The space $\cL(s)\cap\cL(s')$ is isomorphic as a locally convex space to a complemented subspace of $\cL(s)\times\cL(s')$.
\end{Prop}

\bproof
We use matrix representations \eqref{matrix-representation}. If $x=(x_{ij})_{i,j\in\N}\in\Lambda(\cB)$ and $y=(y_{ij})_{i,j\in\N}\in\Lambda(\cB')$ then we denote $M(x,y)\in\C^{\N\times\N}$,
\[[M(x,y)]_{ij}:=\begin{cases}
           x_{ij},\,\,i\leqslant j \\
           y_{ij},\,\,i>j,
          \end{cases}\]
and define a map $P\c\Lambda(\cB)\times\Lambda(\cB')\to\Lambda(\cB)\times\Lambda(\cB')$ by
\[P(x,y):=(M(x,y),M(x,y)).\]
It is easily seen that $P$ is a projection with
\[\operatorname{im}P=\Delta(\Lambda(\cB)\times\Lambda(\cB')):=\{(x,x)\in\Lambda(\cB)\times\Lambda(\cB')\c x\in\Lambda(\cB)\cap\Lambda(\cB')\}.\]

To get continuity, observe first that $\Lambda(\cB)$ and $\Lambda(\cB')$ are webbed by \cite[Lemma 24.28]{MV}. Moreover, $\cL(s)\cong\cL(s')\cong s\otimes s'$ by \cite[\S43, 3.(7)]{K2},
and therefore $\cL(s)$ and $\cL(s')$ are ultrabornological by \cite[Ch. II, Prop. 15 \& Cor. 2]{AG1} (see also \cite{LW} for a homological proof of this fact).
Since $\Lambda(\cB)\cong\cL(s)$ and $\Lambda(\cB')\cong\cL(s')$, $\Lambda(\cB)$ and $\Lambda(\cB')$ are ultrabornological, as well.
This implies that $\Lambda(\cB)\times\Lambda(\cB')$ is ultrabornological by \cite[Ch. II, 8.2, Cor. 1]{HSch} and has a web by \cite[\S35, 4.(6)]{K2}.
Continuity of $P$ follows now by the Closed Graph Theorem \cite[Th. 24.31]{MV}.

Now, let us consider the map
\[\Phi\colon\ls\times\lds\to\Lambda(\cB)\times\Lambda(\cB'),\quad \Phi(x,y):=(\Phi_1(x),\Phi_2(y)),\]
where $\Phi_1\colon\ls\to\Lambda(\cB)$, $\Phi_1\colon\lds\to\Lambda(\cB')$ are isomorphisms given by $x\mapsto(\langle xe_j,e_i\rangle)_{i,j\in\N}$.
Clearly $\Phi$ is an isomorphism of locally convex spaces and
\[\Phi(\Delta(\ls\times\lds))=\Delta(\Lambda(\cB)\times\Lambda(\cB')),\]
where
\[\Delta(\ls\times\lds):=\{(x,\widetilde{x})\in\ls\times\lds\colon x\in\lslds\}.\]
This shows that the map $\Phi^{-1}P\Phi$ is a continuous projection onto $\Delta(\ls\times\lds)$.
Finally, by comparing fundamental systems of seminorms on $\lslds$ and $\Delta(\ls\times\lds)$, we see that the map
\[\Psi\colon\lslds\to\Delta(\ls\times\lds),\quad \Psi(x):=(x,\widetilde{x})\]
is an isomorphisms, and thus $\lslds$ is isomorphic to $\Delta(\ls\times\lds)$, a complemented subspace of $\ls\times\lds$.
\bqed

\begin{Cor}\label{cor_nuclear_ultraborn_PLS}
The space $\lslds$ is a nuclear, ultrabornological $\mathrm{PLS}$-space.
\end{Cor}
\bproof
$\ls$ and $\lds$ are nuclear, ultrabornological $\mathrm{PLS}$-spaces, so is $\ls\times\lds$ as their product (see \cite[Prop. 28.7]{MV}, \cite[Cor. 6.2.14]{PCB}). The desired properties are inherited by complemented subspaces (see \cite[Prop. 28.6]{MV}, \cite[Ch. II, 8.2, Cor. 1]{HSch}, \cite[Prop. 1.2]{DV}), and thus, by Proposition \ref{complemented}, the proof is complete.
\bqed




Let $\cA$ be as in (\ref{plb-matrix}), that is,
\[\cA:=(a_{ij,N,n})_{i,j\in\N,N,n\in\N_0},\hspace{1cm}a_{ij,N,n}:=\max\bigg\{\frac{i^N}{j^n},\frac{j^N}{i^n}\bigg\}.\]

\begin{Prop}\label{prop_lslds=LambdaA}
We have $\lslds\cong\Lambda(\cA)$ as topological ${}^*$-algebras.
\end{Prop}
\bproof
The map
\[T\c\Lambda(\cA)\to\lslds,\,\,\,\,\,\,\,\,\langle(Tx)e_j,e_i\rangle:=x_{ij}\]
is a $^*$-algebra isomorphism. To see it is continuous, observe that the embeddings $\Lambda(\cA)\hk\cL(s)$ and $\Lambda(\cA)\hk\cL(s')$ are continuous. Continuity of $T^{-1}$ follows from the Open Mapping Theorem \cite[24.30]{MV} since
$\Lambda(\cA)$ -- as a $\mathrm{PLS}$-space --  has a web and $\lslds$ is ultrabornological by Corollary \ref{cor_nuclear_ultraborn_PLS}.
\bqed

Combining Proposition \ref{prop_lstars=lslds_as_top_star_alg}, Theorem \ref{th_Ls=DC} and Proposition \ref{prop_lslds=LambdaA} we obtain the following.

\begin{Cor}\label{cor_DC=lstars=lslds=LambdaA}
We have $\mathcal{DC}(\cS)\cong\lstars=\lslds\cong\Lambda(\cA)$ as topological ${}^*$-algebras.
\end{Cor}

From now on, by $\mathcal{ML}(s',s)$ we denote any topological ${}^*$-algebra isomorphic to $\mathcal{DC}(\cS)$ and we call it the \textit{multiplier algebra of the noncommutative Schwartz space}.

We have just shown that $\mathcal{ML}(s',s)$ is webbed and ultrabornological. By \cite[Th. 4.2]{V-proj} this last property is equivalent to barrelledness. Therefore by \cite[Ths. 24.30, 24.31]{MV} and \cite[Prop. 4.1.3]{PCB} all the classical functional analytic tools are available for $\mathcal{ML}(s',s)$. For convenience we state this result separately.

\begin{Th}
\label{classical-fa}
Let $X=\mathcal{ML}(s',s)$ be the multiplier algebra of the noncommutative Schwartz space and let $E,F,G$ be a locally convex spaces with $E$ webbed and $F$ ultrabornological. The following hold:
\begin{enumerate}[\upshape(1)]
\item Uniform Boundedness Principle: every pointwise bounded set $B\subset\mathcal{L}(X,G)$ is equicontinuous.
\item Closed Graph Theorem: every linear map $T\c X\to E$ and $S\c F\to X$ with closed graph is continuous.
\item Open Mapping Theorem: every continuous linear surjection $T\c E\to X$ and $S\c X\to F$ is open.
\end{enumerate}
\end{Th}

We will now need a characterization of those Fr\'echet spaces which are $\mathrm{PLN}$-spaces. This characterization seems to be known for specialists however we were not able to find any reference to that result. For convenience of the reader we state it explicitly below.

\begin{Prop}
\label{strongly}
A Fr\'echet space is a $\mathrm{PLN}$-space if and only if it is strongly nuclear.
\end{Prop}

\bproof
Let $X$ be a strongly nuclear Fr\'echet space. By \cite[21.8, Th. 8]{J}, $X$ is a topological subspace of $(s')^I$ for a countable set $I$. Since $s'$ is an $\mathrm{LN}$-space,
the product $(s')^I$ is a $\mathrm{PLN}$-space. By \cite[Prop. 1.2]{DV}, $X$ is a $\mathrm{PLN}$-space.

Conversely, every $\mathrm{PLN}$-space is a topological subspace of a countable product of $\mathrm{LN}$-spaces and these are strongly nuclear by \cite[21.8, Th. 6]{J}.
Consequently, by \cite[Props. 21.1.3 and 21.1.5]{J}, $X$ is strongly nuclear.
\bqed

\begin{Cor}
\label{topological}
The multiplier algebra of the noncommutative Schwartz space is not a $\mathrm{PLN}$-space.
\end{Cor}

\bproof
Suppose that the multiplier algebra is a $\mathrm{PLN}$-space. Then so is $s$ as its closed subspace. In fact, $s$ is even complemented in $\Lambda(\cA)$ -- consider the projection in $\Lambda(\cA)$
which cancels all but first row-entries. By Proposition \ref{strongly},
$s$ is strongly nuclear which leads to a contradiction by \cite[21.8, Ex. 3]{J}.
\bqed


The multiplier algebra of the noncommutative Schwartz space satisfies another useful property. Recall from \cite[p. 433]{BD} that a $\mathrm{PLS}$-space $X$ is said to have {\it the dual interpolation estimate for big $\theta$} if

\begin{multline}
\forall\,N\,\,\exists\,M\,\,\forall\,K\,\,\exists\,n\,\,\forall\,m\,\,\exists\,\theta_0\in(0,1)\,\,\forall\theta\geqslant\theta_0\,\,\exists\,k,C>0\,\,\forall\,x'\in X'_N: \nonumber \\
||x'\circ\iota^M_N||^*_{M,m}\leqslant C(||x'\circ\iota^K_N||^*_{K,k})^{1-\theta}(||x'||^*_{N,n})^{\theta}. \nonumber
\label{def}
\end{multline}

If we take $\theta\leqslant\theta_0$ then $X$ has the {\it dual interpolation estimate for small $\theta$} and if we take $\theta\in(0,1)$ then $X$ has the {\it dual interpolation estimate for all $\theta$}. For K\"{o}the-type $\mathrm{PLS}$-spaces $\Lambda^p(B)$ it is enough -- see the proof of \cite[Th. 4.3]{BD2} -- to check the above condition for evaluation functionals $\phi_j(x):=x_j,\,j\in\N$. Examples of $\mathrm{PLS}$-spaces with this property can be found in \cite{BD1,BD,BD2}. Dual interpolation estimate plays an important role in partial differential equations, e.g. surjectivity of operators, existence of linear right inverses, parameter dependence of solutions -- see \cite{PD-pardep} for more details.

\begin{Prop}\label{prop_dual_interpolation}
The multiplier algebra $\cM\cS$ has the dual interpolation estimate for big $\theta$ but not for small $\theta$.
\end{Prop}

\bproof
For any $N\in\N_0$ take $M:=N+1$ and for any $K\in\N_0$ take $\theta_0\in(0,1)$ so that
\[(1-\theta_0)K+N\theta_0\leqslant M.\]
In a similar fashion, for $n=1$, any $m\in\N_0$ and $\theta\geqslant\theta_0$ take $k\in\N_0$ so that
\[(1-\theta)k+n\theta\geqslant m.\]
Then (with all the quantifiers in front)
\[\left(\frac{j^K}{i^k}\right)^{1-\theta}\left(\frac{j^N}{i^n}\right)^{\theta}\leqslant C\frac{j^M}{i^m}.\]
Exchanging indices $i,j$ in the above inequality we obtain by \eqref{plb-matrix}
\[(a_{ij;K,k})^{1-\theta}(a_{ij:N,n})^{\theta}\leqslant Ca_{ij;M,m}.\]
Since $\cM\cS\cong\Lambda(\cA)$, the dual interpolation estimate for big $\theta$ follows. If it had the condition for small $\theta$ then, by \cite[Props. 5.3(b) \& 5.4(b)]{BD1} and \cite[Cor. 29.22]{MV}, the space $s$ of rapidly decreasing sequences would be a Banach space -- a contradiction.
\bqed

\begin{Rem}
The above result together with \cite[Prop. 1.1 \& Cor. 1.2(c)]{BD} gives another proof of the fact that the multiplier algebra of the noncommutative Schwartz space is ultrabornological.
\end{Rem}

We end this section with a technical lemma which characterizes when an arbitrary $\mathrm{PLB}$-space is already a $\mathrm{PLS}$-space; the proof uses interpolation theory and follows the idea of \cite[Lemma 7]{KP5}. As a consequence -- see Remark \ref{rem_PLS}(ii), we obtain another proof of the fact that $\cM\cS$ is a $\mathrm{PLS}$-space.

\begin{Lem}
\label{plb-pls}
Let $(X_{M,m})_{M,m\in\N}$ be Banach spaces so that $X:=\opp_{M\in\N}\op_{m\in\N}X_{M,m}$ is a $\mathrm{PLB}$-space. The following conditions are equivalent:
\item[(i)]
$X$ is a $\mathrm{PLS}$-space,
\item[(ii)]
$\forall\,M\,\,\exists\,L:=L_M\,\,\forall\,l\,\,\exists\,m:=m_l\c\,\,X_{L,l}\hk X_{M,m}$\,is a compact inclusion.
\end{Lem}

\bproof
We only need to show the implication (ii)$\Rightarrow$(i). There is no loss of generality in assuming that $L_M=M+1$ and $m_l=l$. Consider the commutative diagram
\begin{center}
\begin{tikzpicture}
  \matrix (m) [matrix of math nodes,row sep=3em,column sep=4em,minimum width=2em]
  {
     X_{2,1} & X_{1,1} \\
     X_{2,2} & X_{1,2} \\};
  \path[right hook-stealth]
    (m-1-1) edge node [left] {$\iota_1^2$} (m-2-1)
            edge node [above] {$j_1$} (m-1-2)
    (m-2-1) edge node [above] {$j_2$} (m-2-2)
    (m-1-2) edge node [right] {$\kappa_1^2$} (m-2-2);
\end{tikzpicture}
\end{center}

\noindent
where the inclusions $\iota_1^2,\kappa_1^2$ are the respective linking maps and the inclusions $j_1,j_2$ are compact. Applying the real interpolation method with parameters $\theta_1,1\,(0<\theta_1<1)$
to the Banach couples $\overline{Y_1}:=(X_{2,1},X_{1,1}),\overline{Y_2}:=(X_{2,2},X_{1,2})$ we obtain, by \cite[Th. 3.11.8]{BL}, a continuous map
\[\mathcal{J}_{\theta_1,1}(\overline{Y_1})\to\mathcal{J}_{\theta_1,1}(\overline{Y_2}).\]
between the interpolation spaces $\mathcal{J}_{\theta_1,1}(\overline{Y_1})$ and $\mathcal{J}_{\theta_1,1}(\overline{Y_2})$. By \cite[Cor. 3.8.2]{BL}, we get for $0<\theta_1<\theta_2<1$ the compact inclusion
\[\mathcal{J}_{\theta_1,1}(\overline{Y_2})\hk\mathcal{J}_{\theta_2,1}(\overline{Y_2})\]
therefore the map
\[j_1^2\c\mathcal{J}_{\theta_1,1}(\overline{Y_1})\to\mathcal{J}_{\theta_2,1}(\overline{Y_2})\]
is also compact. We apply the same procedure to the commutative diagram

\begin{center}
\begin{tikzpicture}
  \matrix (m) [matrix of math nodes,row sep=3em,column sep=4em,minimum width=2em]
  {
     X_{2,2} & X_{1,2} \\
     X_{2,3} & X_{1,3} \\};
  \path[right hook-stealth]
    (m-1-1) edge node [left] {$\iota_2^3$} (m-2-1)
            edge node [above] {$j_2$} (m-1-2)
    (m-2-1) edge node [above] {$j_3$} (m-2-2)
    (m-1-2) edge node [right] {$\kappa_2^3$} (m-2-2);
\end{tikzpicture}
\end{center}

\noindent
and obtain a compact operator
\[j_2^3\c\mathcal{J}_{\theta_2,1}(\overline{Y_2})\to\mathcal{J}_{\theta_3,1}(\overline{Y_3}),\]
where $\overline{Y_3}:=(X_{2,3},X_{1,3})$ and $\theta_2<\theta_3<1$. Proceeding this way we obtain a countable inductive system $(j_n^{n+1}\c\mathcal{J}_{\theta_n,1}(\overline{Y_n})\to Y_1)$,
where $Y_1:=\bigcup_nj_n^{n+1}(\mathcal{J}_{\theta_n,1}(\overline{Y_n}))$. Let us observe that the inductive topology of this system exists. Indeed, let $x\in Y_1$ be a non-zero element.
Since $Y_1\subset X_1=\op_nX_{1,n}$, there exists, by \cite[Lemma 24.6]{MV}, a linear functional $\phi\in(\op_nX_{1,n})^*$ such that $\phi(x)\neq0$ and $\phi\circ\kappa_n\in X_{1,n}'$ for all $n\in\N$ (by $(\kappa_n\c X_{1,n}\to X_1)$ we denote the imbedding spectrum of $X_1$). Recall that we distinguish here the space of linear functionals -- denoted by $(\cdot)^*$ and the space of linear and continuous functionals -- denoted by $(\cdot)'$. Therefore $\phi\in Y_1^*$. Moreover, for every $n\in\N$ we have the commutative diagram

\begin{center}
\begin{tikzpicture}
  \matrix (m) [matrix of math nodes,row sep=3em,column sep=3em,minimum width=2em,text depth=0.5ex,text height=2ex]
  {
     \mathcal{J}_{\theta_n,1}(\overline{Y_n}) & X_{1,n} & \K \\
     \mathcal{J}_{\theta_{n+1},1}(\overline{Y_{n+1}}) & X_{1,n+1} & \K \\};
  \path[-stealth]
    (m-1-1) edge node [left] {$j_n^{n+1}$} (m-2-1)
            edge [right hook-stealth] (m-1-2)
    (m-2-1) edge [right hook-stealth] (m-2-2)
    (m-1-2) edge node [above] {$\phi\circ\kappa_n$} (m-1-3)
    (m-2-2) edge node [above] {$\phi\circ\kappa_{n+1}$} (m-2-3)
    (m-1-3) edge node [right] {$\cong$} (m-2-3);
\end{tikzpicture}
\end{center}

\noindent
therefore $\phi\circ j_n^{n+1}\in\mathcal{J}_{\theta_n,1}(\overline{Y_n})'$ for every $n\in\N$. Again, by \cite[Lemma 24.6]{MV}, this implies that the inductive topology of $(j_n^{n+1}\c\mathcal{J}_{\theta_n,1}(\overline{Y_n})\to Y_1)$ exists. Now, by \cite[Lemma 24.34]{MV}, we conclude that $(j_n^{n+1}\c\mathcal{J}_{\theta_n,1}(\overline{Y_n})\to Y_1)$ is an $\mathrm{LB}$-space and compactness of the linking maps $(j_n^{n+1} )_n$ implies
that it is even an $\mathrm{LS}$-space. It follows that we have linear and continuous maps

\begin{equation}
X_2\to Y_1\to X_1.
\label{plb-pls-eq-1}
\end{equation}

\noindent
Indeed, since for every $n\in\N,\,\mathcal{J}_{\theta_n,1}(\overline{Y_n})$ is an interpolation space for the couple $\overline{Y_n}=(X_{2,n},X_{1,n})$ with the compact inclusion
$j_n\c X_{2,n}\hk X_{1,n}$, we get compact inclusions
\[X_{2,n}\hk\mathcal{J}_{\theta_n,1}(\overline{Y_n})\hk X_{1,n}\]
(observe that we loose injectivity in \eqref{plb-pls-eq-1} because the linking maps $(j_n^{n+1} )_n$ are not, in general, injective).
The above argument works for all $\mathrm{LB}$-spaces $X_M$ and $X_{M+1}$ therefore
\[X=\opp_{M\in\N}Y_M\]
where all $Y_M$'s are $\mathrm{LS}$-spaces. Consequently, $X$ is a $\mathrm{PLS}$-space.
\bqed

\begin{Rem}\label{rem_PLS}
Here we give two new proofs of the fact that $\cM\cS$ is a $\mathrm{PLS}$-space.  By Propositions \ref{lp-equal} \& \ref{prop_lslds=LambdaA}, we may use the topological $^*$-algebra isomorphism $\cM\cS\cong\Lambda(\cA)$ (recall that the $\mathrm{PLB}$-matrix $\cA$ is given by \eqref{plb-matrix}).

(i) We do a slight perturbation of $\cA$. Let $\cD:=(d_{ij,N,n})_{i,j\in\N,N,n\in\N_0}$ be a 4-indexed K\"{o}the $\mathrm{PLB}$-matrix given by
\[d_{ij;N,n}:=\max_{i,j\in\N}\bigg\{\frac{j^{N+\frac{1}{n}}}{i^n},\frac{i^{N+\frac{1}{n}}}{j^n}\bigg\}.\]
One can easily show that for all $i,j\in\N,N,n\in\N_0$ we have
\[a_{ij;N,n}\leqslant d_{ij;N,n}\leqslant a_{ij;N+1,n}.\]
This implies the topological isomorphism $\cM\cS\cong\Lambda(\cD)$. Since
\[\frac{d_{ij;N,n}}{d_{ij;N,n+1}}=\min\{i,j\}\max\{i,j\}^{\frac{1}{n(n+1)}},\]
we get that for all $N,n\in\N_0$
\[\lim_{i,j\to+\infty}\frac{d_{ij;N,n+1}}{d_{ij;N,n}}=0.\]
Consequently, $\Lambda(\cD)$, and therefore also $\cM\cS$, is a $\mathrm{PLS}$-space.

(ii) Since $\frac{a_{ij,N,n+2}}{a_{ij,N+2,n}}=(ij)^{-2}$ for all $i,j,N,n$, the inclusion map
\[\ell_2\big((a_{ij,N+2,n})_{i,j\in\N}\big)\hk\ell_2\big((a_{ij,N,n+2})_{i,j\in\N}\big)\]
is compact. By Lemma \ref{plb-pls}, the result follows.
\bqed
\end{Rem}

\section{Algebraic properties of the multiplier algebra}

We say that a subalgebra $B$ of an algebra $A$ is \emph{spectral invariant in $A$} if, for every $x\in B$,
$x$ is invertible in $A$ if and only if it is invertible in $B$. We show that $\cM\cS$ contains a spectral invariant copy of the algebra $s'$, which implies that
it is neither a $\cQ$-algebra nor $m$-convex.

\begin{Prop}\label{prop_Delta}
Let $\cA$ be a K\"{o}the $\mathrm{PLB}$-matrix given by \eqref{plb-matrix} and let $\Delta(\cA)$ be the algebra of all diagonal matrices belonging to $\Lambda(\cA)$. Then
\begin{enumerate}
 \item[\upshape{(i)}] $\Delta(\cA)$ is a complemented subspace of $\Lambda(\cA)$;
 \item[\upshape{(ii)}] $\Delta(\cA)$ is a closed commutative ${}^*$-subalgebra of $\Lambda(\cA)$;
 \item[\upshape{(iii)}] $\Delta(\cA)\cong s'$ as topological ${}^*$-algebras;
 \item[\upshape{(iv)}] $\Delta(\cA)$ is spectral invariant in $\Lambda(\cA)$.
\end{enumerate}
\end{Prop}

\begin{proof}
(i) Define $\pi\colon\Lambda(\cA)\to\Lambda(\cA)$
by
\[\pi x:=\sum_{j=1}^{\infty}e_{jj}xe_{jj},\]
where $(e_{ij})_{i,j\in\N}$ is a sequence of matrix units.
Clearly, $\pi$ is a projection.
Note that continuity of a linear operator $T$ on $\Lambda(\cA)$ follows from to the condition
\[\forall N\;\exists M\;\forall m\;\exists n,C>0\;\forall x\in\Lambda(\cA)\quad ||Tx||_{N,n}\leqslant C||x||_{M,m}.\]
But
\[||\pi x||_{N,n}=\sum_{j=1}^\infty|x_{jj}|a_{jj,N,n}\leqslant\sum_{i,j=1}^\infty|x_{ij}|a_{ij,N,n}=||x||_{N,n},\]
so $\pi$ is continuous, and thus $\Delta(\cA)$ is complemented in $\Lambda(\cA)$.

(ii) It is clear that $\Delta(\cA)$ is a commutative ${}^*$-subalgebra of $\Lambda(\cA)$, and by (i) it is closed in $\Lambda(\cA)$.

(iii) Since
\begin{align*}
\Delta(\cA)&=\{x\in\C^{\N\times\N}\colon x_{ij}=0 \text{ for }i\neq j \text{ and }\forall N\;\exists n\quad\sum_{j=1}^\infty|x_{jj}|j^{N-n}<\infty\}\\
&=\{x\in\C^{\N\times\N}\colon x_{ij}=0 \text{ for }i\neq j \text{ and }(x_{jj})_{j\in\N}\in s'\},
\end{align*}
the operator
\[\phi\colon s'\to \Delta(\cA),\,\,\,\,\,\,\,\,\phi\xi:=\sum_{j=1}^{\infty}\xi_je_{jj}\]
is a ${}^*$-algebra isomorphism. Moreover, for all $N,n\in\N_0$ and all $\xi\in s'$, we have
\[||\phi\xi||_{N,n}=\sum_{j=1}^\infty|\xi_j|j^{N-n}.\]
Now, by Cauchy-Schwarz we get
\[|\xi|_{-m}\leqslant\sum_{j=1}^\infty|\xi_j|j^{-m}\leqslant\frac{\pi}{\sqrt{6}}|\xi|_{-m+1}\]
for all $m\in\N_0$. Therefore $\phi$ is an isomorphism of locally convex spaces. Consequently, $\Delta(\cA)\cong s'$ as topological ${}^*$-algebras.

(iv) Let us take $x\in\Delta(\cA)$ which is invertible in $\Lambda(\cA)$ and let $y$ be its inverse. Then
\[
\sum_{k=1}^\infty x_{ik}y_{kj}=x_{ii}y_{ij}= \left\{\begin{array}{ll}
1 & \textrm{for $i=j$},\\
0 & \textrm{otherwise}.
\end{array} \right.
\]
Consequently, $x_{ii}\neq0$, and thus $y_{ij}=0$ for $i\neq j$. This shows that $y\in\Delta(\cA)$, and the proof is complete.
\end{proof}

\begin{Prop}
The following statements hold:
\begin{enumerate}
 \item[\upshape(i)] $\cM\cS$ is not a $\cQ$-algebra;
 \item[\upshape(ii)] $\cM\cS$ is not $m$-convex.
\end{enumerate}
\end{Prop}

\bproof
By Proposition \ref{prop_Delta}, $\mathcal{ML}(s',s)$ contains a closed, spectral invariant ${}^*$-subalgebra $M$ isomorphic to $s'$.

(i) Let $F,G$ be the sets of invertible elements in $M$ and $\mathcal{ML}(s',s)$, respectively.
Then $F=G\cap M$. This shows that $G$ is not open, because otherwise $F$ would be open, which contradicts \cite[Th. 2.8]{BonDom}. Hence $\mathcal{ML}(s',s)$ is not a $\cQ$-algebra.

(ii) Suppose that there is a basis $\cV$ of zero neighborhoods in $\mathcal{ML}(s',s)$ such that $V^2\subseteq V$ for all $V\in\cV$.
Then $\{V\cap M\}_{V\in\cV}$ is a basis of zero neighborhoods in $M$ and
\[(V\cap M)^2\subseteq V^2\cap M\subseteq V\cap M,\]
so $M$ is $m$-convex, a contradiction (again apply \cite[Th. 2.8]{BonDom}). Hence, $\mathcal{ML}(s',s)$ is not $m$-convex.
\bqed

\printbibliography


%
%
%
%
%
%
%
%
%
%

\vspace{2cm}

\begin{minipage}{0.7\textwidth}
Tomasz Cia\'s

Faculty of Mathematics and Computer Science

Adam Mickiewicz University, Pozna{\'n}

ul. Uniwersytetu Pozna{\'n}skiego 4

61-614 Pozna{\'n}, Poland

e-mail: tcias@amu.edu.pl
\end{minipage}

\vspace{1cm}

\begin{minipage}{0.7\textwidth}
Krzysztof Piszczek

Faculty of Mathematics and Computer Science

Adam Mickiewicz University, Pozna{\'n}

ul. Uniwersytetu Pozna{\'n}skiego 4

61-614 Pozna{\'n}, Poland

e-mail: kpk@amu.edu.pl
\end{minipage}

\end{document}